\documentclass[12pt]{amsart}

\usepackage{tgpagella}
\usepackage{euler}
\usepackage[T1]{fontenc}
\usepackage{amsmath, amssymb}
\usepackage[hidelinks]{hyperref}
\usepackage[english]{babel}
\usepackage{mathrsfs}
\usepackage{eucal}
\usepackage[all]{xy}
\usepackage{tikz}

\newtheorem{thm}{Theorem}[section]
\newtheorem*{thm*}{Theorem}
\newtheorem{lem}[thm]{Lemma}
\newtheorem{fact}[thm]{Fact}

\newtheorem{prop}[thm]{Proposition}
\newtheorem*{prop*}{Proposition}

\newtheorem{cor}[thm]{Corollary}
\newtheorem*{cor*}{Corollary}

\theoremstyle{definition}
\newtheorem{defn}[thm]{Definition}
\newtheorem*{defn*}{Definition}

\newtheorem{remark}[thm]{Remark}

\newtheorem*{question*}{Question}
\newtheorem*{Pquestion*}{Popa's question}

\newtheorem*{conv*}{Convention}

\def\cal{\mathcal}

\def\fin{\Subset}

\def \starR{{}^{*}\mathbb R}

\makeatletter

\def\dotminussym#1#2{%
  \setbox0=\hbox{$\m@th#1-$}%
  \kern.5\wd0%
  \hbox to 0pt{\hss\hbox{$\m@th#1-$}\hss}%
  \raise.6\ht0\hbox to 0pt{\hss$\m@th#1.$\hss}%
  \kern.5\wd0}

\DeclareMathOperator{\id}{id}

\def \cirB{{}^{\circ}B}
\def \cirA{{}^{\circ}A}
\def \cirH{{}^{\circ}H}
\def \cirP{{}^{\circ}P}
\def \cirx{{}^{\circ}x}
\def \ciry{{}^{\circ}y}
\def \fin{\operatorname{fin}}
\def \st{\operatorname{st}}
\def \bst{\boldsymbol{\operatorname{st}}}
\def \bor{\operatorname{Borel}}
\def \loeb{\operatorname{Loeb}}
\def \id{\operatorname{id}}

\begin{document}

%%%%%%%%%%%%%%%%%%%%%%%%%%%%%%%%%%%%%%%%%%%%%%

\title[A nonstandard proof of the spectral theorem for unbounded operators]{A nonstandard proof of the spectral theorem for unbounded self-adjoint operators}
\author{Isaac Goldbring}
\thanks{The author was partially supported by NSF CAREER grant DMS-1349399.}

\address{Department of Mathematics\\University of California, Irvine, 340 Rowland Hall (Bldg.\# 400),
Irvine, CA 92697-3875}
\email{isaac@math.uci.edu}
\urladdr{http://www.math.uci.edu/~isaac}

\maketitle

\begin{abstract}
We generalize Moore's nonstandard proof of the Spectral theorem for bounded self-adjoint operators to the case of unbounded operators.  The key step is to use a definition of the nonstandard hull of an internally bounded self-adjoint operator due to Raab.
\end{abstract}

\section{Introduction}

Throughout this note, all Hilbert spaces will be over the complex numbers and (following the convention in physics) inner products are conjugate-linear in the \emph{first} coordinate.

The goal of this note is to provide a nonstandard proof of the \textbf{Spectral theorem for unbounded self-adjoint operators}:

\begin{thm}\label{maintheorem}
Suppose that $A$ is an unbounded self-adjoint operator on the Hilbert space $\cal H$.  Then there is a projection-valued measure $P:\bor(\mathbb R)\to \cal B(\cal H)$ such that $$A=\int_{\mathbb R}\id dP.$$
\end{thm}

Here, $\bor(\mathbb R)$ denotes the $\sigma$-algebra of Borel subsets of $\mathbb R$ and $\id$ denotes the identity function on $\mathbb R$.  All of the terms appearing in the previous theorem will be defined precisely in the next section.  We refer to the expression $A=\int_{\mathbb R}\id dP$ as a \textbf{spectral resolution} of $A$.  Thus, the theorem says that every unbounded self-adjoint operator admits a spectral resolution.  In fact, this spectral resolution is unique (that is, the projection-valued measure yielding the resolution is unique), although we will not address the uniqueness issue here.  The Spectral theorem for unbounded operators is especially important in quantum mechanics, for it provides one with the probability distributions for measuring observables with continuous spectra such as position and momentum.

The simplest version of the above Spectral theorem is the case when $\cal H$ is finite-dimensional.  In that case, there is a finite set $\lambda_1,\ldots,\lambda_n$ of distinct eigenvalues of $A$ such that, denoting the projection onto the eigenspace corresponding to $\lambda_i$ by $P_i$, the spectral resolution of $A$ is given by $A=\sum_{i=1}^n \lambda_i P_i$.  

In \cite{moore}, Moore gave a very simple nonstandard proof of the Spectral theorem for \emph{bounded} self-adjoint operators.  The idea of the proof is as follows.  (This is not literally how Moore phrased things but what follows is morally equivalent to his argument.)  First, one embeds $\cal H$ in a \emph{hyperfinite-dimensional} Hilbert space $H$ and extends $A$ to an internally bounded operator $B$ on $H$.  By transferring the finite-dimensional version of the Spectral theorem stated in the previous paragraph, we have an internal spectral resolution of $B$ into a hyperfinite sum of eigenvalues multiplied by the projections onto the corresponding eigenspaces.  By quotienting out by elements of infinitesimal norm, one obtains a standard Hilbert space $\cirH$, called the \textbf{nonstandard hull} of $H$, which still contains the original Hilbert space $\cal H$.  It then follows from the fact that the internal operator norm of $B$ is finite that $B$ descends to a bounded linear operator $\cirB$ on $\cirH$, called the nonstandard hull of $B$, which still extends $A$.  Moreover, one can convert the internal projection-valued measure which yields the spectral resolution of $B$ into an actual projection-valued measure except that its domain is not $\bor(\mathbb R)$, but rather $\loeb(\starR)$, the $\sigma$-algebra generated by the internally Borel subsets of $\starR$.  However, by ``pushing forward'' this projection-valued measure using the standard part map, one obtains an actual projection-valued measure, which one can then check yields a spectral resolution of $\cirB$.  By composing this projection-valued measure with the projection mapping $\cirH$ onto $\cal H$, one obtains a spectral resolution of the original operator $A$.  (The reader familiar with nonstandard integration theory might notice the similarities  between this construction and the nonstandard construction of the Lebesgue integral from an appropriate hyperfinite sum.)

It is the purpose of this note to extend the argument in the preceding paragraph to the case that the operator $A$ is unbounded.  The main difficulty in extending Moore's argument is the question of how one defines the map $\cirB$ on $\cirH$.  However, in \cite{raab}, Raab develops a nonstandard approach to quantum mechanics in which a definition of $\cirB$ is given for internally bounded (but not necessarily finitely bounded) self-adjoint operators $B$.  Equipped with this definition, it is possible to adapt Moore's argument to the unbounded case.  We stress that essentially all of the ingredients needed in our proof are present in Raab's article.  Our main contribution is simply to point out the applicability of his work to the Spectral theorem for unbounded operators.  In addition, we take this opportunity to give more detailed proofs of claims made in \cite{raab} for which we felt (in our humble opinion) extra detail would be warranted.

We believe that the nonstandard proof of Theorem \ref{maintheorem} is technically easier and pedagogically superior to the classical proof.  Classically, Theorem \ref{maintheorem} is proven by using the \textbf{Cayley transform} to reduce to the case of the Spectral theorem for bounded \emph{normal} operators.  The proof for normal operators offers some challenges not apparent in the self-adjoint case and the use of the Cayley transform can be hard to motivate.  However, the nonstandard proof of the unbounded case only relies on the \emph{finite-dimensional} case and some fairly elementary nonstandard integration theory.  In particular, it does not rely on any version of the bounded case to be proven first.

In Section 2, we recall some basic facts about unbounded operators and projection-valued measures.  Of particular importance to us are the notions of restrictions of projection-valued measures and pushforwards of projection-valued measures.  While these notions are quite natural, we were unable to find precise references for them in the literature and so we take this opportunity to spell out some details of these constructions.  This section also contains the requisite nonstandard measure and integration theory needed for the rest of the paper.  Section 3 contains a precise proof of the bounded case along the lines of the sketch above.  We reiterate that our proof of the unbounded situation does not rely on the bounded case; we merely include the proof of the bounded case as motivation for Section 4, where the nonstandard hull of an internally bounded self-adjoint operator is defined and where two important facts about this definition are proven.  Finally, in Section 5, we use the material in Section 4 to adapt Moore's argument to the unbounded case.

In order to keep this note somewhat short, we assume that the reader is familiar with some basic nonstandard analysis.  The unacquainted reader can consult \cite{gold} for a quick introduction.  (The reader will notice that in \cite{gold} nonstandard extensions are adorned with stars on the right-hand side whereas in this article the stars are on the more traditional left-hand side; the reason for this is due to the large number of appearances of adjoint operators in this article.)

We would like to thank C. Ward Henson and David Ross for their help in finding useful references and to Giorgio Metafune and Andreas Raab for helpful discussions regarding this work.

\section{Some reminders, standard and otherwise}

\subsection{Unbounded operators}

Fix a Hilbert space $\cal H$.  An \textbf{unbounded operator} on $\cal H$ is a linear operator $A:D(A)\to \cal H$, where $D(A)$ is a dense subspace of $\cal H$ called the \textbf{domain} of $A$. (We follow the unfortunate convention that an unbounded operator may indeed be bounded, that is, unbounded really means not necessarily bounded.) If $B$ is also an unbounded operator on $\cal H$, we say that $B$ is an \textbf{extension} of $A$, written $A\subseteq B$, if $D(A)\subseteq D(B)$ and $Bx=Ax$ for all $x\in D(A)$.  In fact, we extend this terminology and notation to cover the case that $A$ is an unbounded operator on $\cal H$, $B$ is an unbounded operator on $\cal K$, where $\cal H$ is a closed subspace of $\cal K$, and $D(A)\subseteq D(B)$ and $Bx=Ax$ for all $x\in D(A)$. 

If $A$ is an unbounded operator on $\cal H$, we define the \textbf{adjoint of $A$} to be the linear operator $A^*$ defined as follows.  The domain $D(A^*)$ of $A^*$ consists of all $x\in \cal H$ for which the map $y\mapsto \langle x,Ay\rangle:D(A)\to \mathbb{C}$ is bounded, meaning that there is $K\geq 0$ such that $|\langle x,Ay\rangle|\leq K\|y\|$ for all $y\in D(A)$.  If $x\in D(A^*)$, then $A^*x$ is the unique element of $\cal H$ for which $\langle A^*x,y\rangle=\langle x,Ay\rangle$ for all $y\in D(A)$.  It is not always the case that $A^*$ is an unbounded operator on $\cal H$, that is, $D(A^*)$ is not always dense in $\cal H$; this happens if and only if $A$ is \textbf{closable}, meaning the closure of the graph of $A$ (as a subset of $\cal H\oplus \cal H$) is the graph of an unbounded operator on some (necessarily dense) subspace of $\cal H$.  For the unbounded operators of interest in this paper (namely the symmetric and self-adjoint operators defined in the next paragraph), this is always the case.

The unbounded operator $A$ on $\cal H$ is said to be \textbf{self-adjoint} if $A=A^*$ as unbounded operators on $\cal H$.  (This includes assuming that $D(A)=D(A^*)$.)  An unbounded operator $A$ is \textbf{symmetric} if $\langle Ax,y\rangle=\langle x,Ay\rangle$ for all $x,y\in D(A)$.  Equivalently, $A$ is symmetric if and only if $A^*$ is an extension of $A$.  In particular, self-adjoint operators are symmetric.  On the other hand, a symmetric operator $A$ is self-adjoint if and only if $D(A)=D(A^*)$.  

The following fact will be used later on in this note; see \cite[Proposition X.2.11(c)]{conway}.

\begin{fact}\label{noextensions}
Self-adjoint operators have no proper symmetric extensions (on the same Hilbert space).
\end{fact}

If $B$ is a self-adjoint unbounded operator on some Hilbert space $\cal K$ and $\cal H$ is a closed subspace of $\cal K$, we say that $\cal H$ is \textbf{reducing for $B$} if $P_{\cal H}B\subseteq BP_{\cal H}$, where $P_{\cal H}\in \cal B(\cal K)$ is the orthogonal projection onto $\cal H$.  To be precise, this means:  if $x\in D(B)$, then $P_{\cal H}x\in D(B)$ and $P_{\cal H}(Bx)=B(P_{\cal H}x)$. If $\cal H$ is reducing for $B$, then we may define the unbounded operator $B|\cal H$ on $\cal H$ with domain $D(B|\cal H):=D(B)\cap \cal H$ and which acts like $B$ on its domain.  It is routine to verify that $B|\cal H$ is a self-adjoint operator on $\cal H$.

The following is probably well-known, but since we could not find a reference, we include a proof here:

\begin{prop}\label{selfadjointext}
Suppose that $\cal H$ is a closed subspace of $\cal K$ and that $A$ and $B$ are self-adjoint operators on $\cal H$ and $\cal K$ respectively for which $A\subseteq B$.  Then $\cal H$ is reducing for $B$ and $A=B|\cal H$.
\end{prop}

\begin{proof}
First, fix $x\in D(B)$; we show that $P_{\cal H}x\in D(B)$.  By the assumptions of the proposition, it suffices to show that $P_{\cal H}x\in D(A^*)$.  Since $x\in D(B^*)$, we know that the map $y\mapsto \langle x,By\rangle$ is bounded on $D(B)$ and thus on $D(A)$ as well.  Since $\langle P_{\cal H}x,Ay\rangle=\langle x,Ay\rangle=\langle x,By\rangle$ for all $y\in D(A)$, we see that $y\mapsto \langle P_{\cal H}x,Ay\rangle$ is bounded on $D(A)$ and thus $P_{\cal H}x\in D(A^*)$, as desired.

We continue to assume that $x\in D(B)$ and show that $P_{\cal H}(Bx)=B(P_{\cal H}x)$.  By the argument in the previous paragraph, we know that $P_{\cal H}x\in D(A)$, so $B(P_{\cal H}x)=A(P_{\cal H}x)$.  Fix $y\in D(A)$.  We then have
$$\langle y,Bx-A(P_{\cal H}x)\rangle=\langle By,x-P_{\cal H}x\rangle=\langle Ay,x-P_{\cal H}x\rangle=0.$$
Since $D(A)$ is dense in $\cal H$, we see that $Bx-A(P_{\cal H}x)\in \cal H^\perp$, whence it follows that $P_{\cal H}(Bx)=A(P_{\cal H}x)=B(P_{\cal H}x)$, as desired.

We have shown that $\cal H$ is reducing for $B$, whence $B|\cal H$ is a self-adjoint operator on $\cal H$.  Since $B|\cal H$ clearly extends $A$, we conclude by Fact \ref{noextensions}.
\end{proof}

\subsection{Projection-valued measures}

\begin{defn}
Fix a measurable space $(X,\frak B)$.  A \textbf{projection-valued measure} on $(X,\frak B)$ is a function $P:\frak B\to \cal B(\cal H)$, where $\cal H$ is a Hilbert space, satisfying the following properties:
\begin{enumerate}
    \item For each $\Omega\in \frak B$, $P_\Omega:=P(\Omega)$ is a projection on $\cal H$.
    \item $P_\emptyset=0$ and $P_X=I$.
    \item If $(\Omega_n)_{n\in \mathbb N}$ is a collection of disjoint elements of $\frak B$ and $\Omega=\bigcup_{n\in \mathbb N}\Omega_n$, then for all $x\in \cal H$, we have $P_\Omega(x)=\sum_{n\in \mathbb N}P_{\Omega_n}(x)$.
\end{enumerate}
% If $\frak B$ is clear from context, then we speak simply of a projection-valued measure on $X$.  
\end{defn}

Fix a projection-valued measure $P:\frak B\to \cal B(\cal H)$ on $X$.  For all $x,y\in \cal H$, one can consider the complex measure $\mu^{x,y}$ on $X$ given by $\mu^{x,y}(\Omega):=\langle y,P_{\Omega}(x)\rangle$.  Moreover, $|\mu^{x,y}|\leq \|x\|\|y\|$; here $|\mu^{x,y}|$ denotes the total variation of $\mu^{x,y}$.  (See \cite[Lemma IX.1.9]{conway}.)  We set $\mu^x:=\mu^{x,x}$ and note that it is a positive measure on $X$ with $\mu^x(X)=\|x\|^2$.  

Given a measurable function $f:X\to \mathbb C$, one can define an unbounded operator $\int_X fdP$ on $\cal H$ as follows.  First, we set $D(\int_X fdP):=\{x\in \cal H \ : \ \int_X |f|^2d\mu^x<\infty\}$.  One can check that $D(\int_X fdP)$ is indeed a dense subspace of $\cal H$.  For $x\in D(\int_X fdP)$, one defines $(\int_X fdP)x$ so that the formula $\langle y,(\int_X fdP)x\rangle=\int_X fd\mu^{x,y}$ holds for all $y\in \cal H$.  We note that if $x\in D(\int_X fdP)$, then $f\in L^1(|\mu^{x,y}|)$ for all $y\in \cal H$.  Also, if $f$ is real-valued, then $\int_X fdP$ is a self-adjoint operator on $\cal H$.  (See \cite[Section X.4]{conway} for proofs of these facts.)

If $A$ is a self-adjoint operator on $\cal H$, then a \textbf{spectral measure for $A$} is a projection-valued measure $P:\bor(\mathbb R)\to \cal B(\cal H)$ such that $A=\int_{\mathbb R} \id dP$ (as unbounded operators on $\cal H$); we refer to the latter equation as a \textbf{spectral resolution of $A$}.  Theorem \ref{maintheorem} thus asserts that every self-adjoint operator admits a spectral measure.

The following result will be crucial for us.  Although it is probably well-known, we were unable to find an exact reference in the literature.  For that reason, we include a brief proof sketch.

\begin{prop}\label{inducedpvm}
Suppose that $B$ is a self-adjoint operator on $\cal K$ and $\cal H$ is a reducing subspace for $B$.  Suppose that $P:\bor(\mathbb R)\to \cal B(\cal K)$ is a spectral measure for $B$ and $P_{\cal H}$ is the projection onto $\cal H$.  Then:
\begin{enumerate}
    \item $P_{\cal H}$ commutes with $P_\Omega$ for every $\Omega\in \bor(\mathbb R)$.
    \item Setting $P|\cal H:\bor(\mathbb R)\to \cal B(\cal H)$ to be the function given by $(P|\cal H)(\Omega):=P_{\cal H}\circ (P_\Omega|\cal H)$, we have that $P|\cal H$ is a spectral measure for $B|\cal H$.
\end{enumerate}
\end{prop}

\begin{proof}[Proof sketch]
Item (1) is usually included as part of the Spectral theorem for unbounded self-adjoint operators and every reference we have found proves this \emph{using} the specific construction for a spectral measure for $B$.  Since we wish to avoid circular reasoning, it is important for us to note that (1) is a formal consequence of the existence of the spectral resolution for $B$ and thus we include some details.  (We thank Giorgio Metafune for assisting us with this proof sketch.)  First, since $B$ is self-adjoint, its spectrum is contained in $\mathbb R$, whence, given any real number $t$, we have that the operator $B-tiI$ has a bounded inverse, denoted $(B-tiI)^{-1}$.  Since $\cal H$ is a reducing subspace for $B$, it is readily verified that the bounded operators $P_{\mathcal H}$ and $(B-tiI)^{-1}$ commute.  Using the formula $e^{itx}=\lim_{n\to \infty}\left(\frac{n}{t}\left(\frac{n}{t}-ix\right)^{-1}\right)^n$, it follows that $P_{\cal H}$ commutes with the operator $e^{itB}$ for any $t\in \mathbb R$.  We note that linear combinations of exponential functions $e^{itx}$, when restricted to a given compact set, are dense in the set of continuous functions on that compact set.  Consequently, we see that $f(B)$ commutes with $P_{\cal H}$ for any continuous function $f$ with compact support.  Standard approximation arguments now allow us to conclude (1).

% Since $\cal H$ is reducing for $B$, using the formula that $e^{itB}=\lim_{n\to \infty} \left(\frac{n}{t}\left(\frac{n}{t}-iB\right)^{-1}\right)^n$, one can show that $e^{ikB}$ commutes with $P_{\cal H}$ for all $k\in \mathbb Z$.
% Since $\cal H$ is reducing for $B$, we see that the bounded operator $(B-\lambda_0)^{-1}$ commutes with $P_{\cal H}$ for every $\lambda_0$ in the resolvent set of $B$, that is, $f(B)$ commutes with $P_{\cal H}$, where $f(t):=(t-\lambda_0)^{-1}$.  Taking $\lambda_0=x\pm \lambda i$, the same holds for $f(t):=\lambda((t-x)^2+\lambda^2)^{-1}$.  
% Since linear combinations of these functions are dense in the set of continuous functions with compact support, the same holds true for such functions, and thus, by approximations again, for characteristic functions of intervals, and finally for Borel sets, as desired.
By (1), we have that $P_{\cal H}\circ (P_\Omega|\cal H)$ is a projection on $\cal H$ for each $\Omega\in \bor(\mathbb R)$.  It is readily verified that $P|\cal H$ is a projection-valued measure on $\cal H$ whose associated complex measures agree with those of $P$ itself (restricted to pairs of elements of $\cal H$).  From this and Fact \ref{noextensions}, it is readily verified that $B|\cal H=\int_{\mathbb R}\id d(P|\cal H)$.
\end{proof}

It is possible to push-forward projection-valued measures along measurable maps.  Since this construction does not appear to be well-documented in the literature, we provide some specifics here.  

\begin{defn}
Suppose that $\Phi:(X,\frak B)\to (Y,\frak C)$ is a measurable map between measurable spaces.  If $P:\frak B\to \cal B(\cal H)$ is a projection-valued measure on $(X,\frak B)$, then we define the \textbf{push-forward} of $P$ via $\Phi$ to be the projection-valued measure $\Phi_*P:\frak C\to\cal B(\cal H)$ on $(Y,\frak C)$ given by setting $(\Phi_*P)_\Omega:=P_{\Phi^{-1}(\Omega)}$ for all $\Omega\in \frak C$.
\end{defn}  

The following lemma is obvious but useful:

\begin{lem}
Suppose that $\Phi:(X,\frak B)\to (Y,\frak C)$ is a measurable map between measurable spaces and $P:\frak B\to \cal B(\cal H)$ is a projection-valued measure on $(X,\frak B)$.  For $x,y\in \cal H$, let $\mu^{x,y}$ and $\nu^{x,y}$ denote the complex measures associated to $P$ and $\Phi_*P$ respectively.  Then $\nu^{x,y}=\Phi_*\mu^{x,y}$.
\end{lem}

In the next lemma (and in a couple of other places in this note), we will need to make use of the \textbf{Polarization Identity}, namely, for any sesquilinear form $g$ on a complex vector space $V$, we have $g(x,y)=\frac{1}{4}\sum_{k=1}^4 i^k g(x+i^ky,x+i^ky)$ for all $x,y\in V$.  We apply this to sesquilinear forms of the form $(x,y)\mapsto \langle y,Ax\rangle$, where $A$ is a linear operator on some complex inner product space.  In particular, we see that if $A$ and $B$ are two linear operators on some complex inner product space $V$ for which $\langle x,Ax\rangle=\langle x,Bx\rangle$ for all $x\in V$, then we in fact have $\langle y,Ax\rangle=\langle y,Bx\rangle$ for all $x,y\in V$.  (Technically, we will sometimes use this identity in the case that $A$ and $B$ are linear transformations from $W$ to $V$, where $W$ is some subspace of a complex inner product space $V$.)

We now show that the usual change of variables formula for pushforward measures holds in the context of projection-valued measures. 

\begin{lem}\label{pushforward}
Suppose that $\Phi:(X,\frak B)\to (Y,\frak C)$ is a measurable map between measurable spaces and $P:\frak B\to \cal B(\cal H)$ is a projection-valued measure on $(X,\frak B)$.  Then for any measurable function $f:Y\to \mathbb R$, we have $\int_Y fd(\Phi_*P)=\int_X (f\circ \Phi)dP$ (as unbounded operators on $\cal H$).
\end{lem}

\begin{proof}
Suppose that $x\in D(\int_Y fd(\Phi_*P))$.  Letting $\mu^x$ denote the positive measure on $(X,\frak B)$ associated with $P$ and $x$, we have $\int_Y \|f\|^2d(\Phi_*\mu^x)<\infty$.  By the usual change of variables formula, we have $\int_X \|f\circ\Phi\|^2d\mu^x<\infty$, so $x\in D(\int_X (f\circ\Phi)dP)$.  For such an $x$, we also have
$$\left\langle x,\left(\int_Y fd(\Phi_*P)\right)x\right\rangle=\int_Y fd(\Phi_*\mu^x)=\int_X (f\circ \Phi)d\mu^x=\left\langle x,\left(\int_X (f\circ \Phi)dP\right)x\right\rangle.$$ By the Polarization Identity, we conclude that
$$\left\langle y,\left(\int_Y fd(\Phi_*P)\right)x\right\rangle=\left\langle y,\left(\int_X (f\circ \Phi)dP\right)x\right\rangle$$ holds for all $x,y\in D(\int_Y fd(\Phi_*P))$.  Since this latter domain is dense in $\cal H$, it follows that the above identity holds true for all $y\in \cal H$, whence $\int_X (f\circ \Phi)dP$ is an extension of $\int_Y fd(\Phi_*P)$.  Since $f$ is real-valued, we have that both operators are self-adjoint, whence we can conclude using Fact \ref{noextensions}.
\end{proof}

% Using this, one can then show that the usual change of variables formula holds for the push-forward of projection-valued measures, namely, for any measurable $f:Y\to \mathbb C$, we have $\int_Y fd(\Phi_*P)=\int_X (f\circ \Phi)dP$ (as unbounded operators on $\cal H$).

\subsection{Loeb measure}

Suppose that $\mu$ is an internal, finitely bounded positive measure on ${}^*\bor(\mathbb R)$, the algebra of internally Borel subsets of $\starR$.  In particular, $\mu(\starR)\in \fin(\starR)$ and $\mu$ is finitely additive.  One can then define a standard, finitely additive, finite positive measure $\st(\mu)$ on ${}^*\bor(\mathbb R)$ given by $\st(\mu)(\Omega):=\st(\mu(\Omega))$.  A saturation argument shows that the Carath\'eodory Extension theorem can be applied to extend $\st(\mu)$ to a genuine positive measure $\mu_L$ on $\loeb(\starR)$, the $\sigma$-algebra generated by ${}^*\bor(\mathbb R)$, called the \textbf{Loeb measure} associated to $\mu$.

It will be important for us to understand when internal functions that are integrable with respect to an internal measure have standard parts that are integrable with respect to the associated Loeb measure $\mu_L$.  In what follows, when we say that an internal function $f:\starR\to \starR$ is internally measurable, we implicitly mean with respect to the algebra ${}^*\bor(\mathbb R)$.  The easiest case in this direction is the following; for a proof, see \cite[Theorem 6.1]{ross}.

\begin{fact}\label{finiteintegral}
Fix an internal finitely bounded positive measure $\mu$ on ${}^*\bor(\mathbb R)$ and suppose that $f:\starR\to \starR$ is an internal measurable function that is finitely bounded, that is, there is $M\in \mathbb R$ such that $|f(x)|\leq M$ for all $x\in \starR$.  Then $\st(f):\starR\to \mathbb R$ is integrable with respect to $\mu_L$ and $\st(\int_{\starR}fd\mu)=\int_{\starR}\st(f)d\mu_L$.
\end{fact}

A routine overflow argument shows that the assumption on $f$ in the previous fact is equivalent to the assumption that $f(x)\in \fin(\starR)$ for all $x\in \starR$.  When an internally measurable function $f:\starR\to \starR$ takes values in $\fin(\starR)$ almost everywhere, but not actually everywhere, the situation is a little more complicated (but still well-understood).  In what follows, we set $\bst:\starR\to \starR$ to be given by $\bst(x)=\st(x)$ if $x\in \fin(\starR)$ and otherwise $\bst(x)=0$.

\begin{defn}
Fix an internal finitely bounded positive measure $\mu$ on ${}^*\bor(\mathbb R)$ and suppose that $f:\starR\to \starR$ is an internally $\mu$-integrable function such that $f(x)\in \fin(\starR)$ $\mu$-almost everywhere.  We say that $f$ is \textbf{S-integrable} if $\bst(f)$ is $\mu_L$-integrable and $\st(\int_{\starR}fd\mu)=\int_{\starR}\bst(f)d\mu_L$.
\end{defn}

The key for us is the following result due to Lindstr\o{}m; see \cite[Theorem 6.3]{ross}.

\begin{fact}\label{Lindstrom}
Fix an internal finitely bounded positive measure $\mu$ on ${}^*\bor(\mathbb R)$ and suppose that $f:\starR\to \starR$ is an internally measurable function such that $|f|^p$ is $\mu$-integrable for some (standard) $p\in (1,\infty)$.  Then $f$ is S-integrable.
\end{fact}

We will really only use the preceding fact in the case that $f(x)=x$ with $p=2$ and $f(x)=x^2$ with $p=1$; both of these cases are fairly easy to prove by hand.  In any event, Lindstr\o{}m's result is not very difficult and our reliance on this result should not count against the simplicity of the nonstandard arguments to follow.

Since measures associated to projection-valued measures are complex, we need a bit of complex Loeb measure/integration theory.  The literature is surprisingly sparse on this topic, so we explicitly spell out exactly what we need.  We call an internal complex measure $\mu$ on ${}^*\bor(\mathbb R)$ finitely bounded if $|\mu|$ is a finitely bounded internal (positive) measure.

\begin{defn}
Suppose that $\mu$ is an internal finitely bounded complex measure on ${}^*\bor(\mathbb R)$ and let $h:=\frac{d|\mu|}{d\mu}$ be the internal Radon-Nikodym derivative.  We then let $d\mu_L$ be the standard complex measure on $\loeb(\starR)$ satisfying $$d\mu_L=\st(h)d|\mu|_L.$$
\end{defn}

Note in the previous definition that $\st(h)$ indeed makes sense since $|h(x)|=1$ for all $x\in \Omega$.  Also note that $|\mu_L|=|\mu|_L$.  (This follows from, for example, \cite[Theorem 6.13]{rudin}.) 

The next lemma shows that our definition is a good one.

\begin{lem}\label{complexmeasure}
Suppose that $\mu$ is an internal finitely bounded complex measure on ${}^*\bor(\mathbb R)$ and $\Omega\in {}^*\bor(\mathbb R)$.  Then $\mu_L(\Omega)=\st(\mu(\Omega))$.
\end{lem}

\begin{proof}
We calculate
$$\mu_L(\Omega)=\int_{\starR}\chi_\Omega d\mu_L=\int_{\starR}\chi_\Omega\st(h)d|\mu|_L=\st\left(\int_{\starR}\chi_\Omega hd|\mu|\right)=\st(\mu(\Omega)),$$ where the next-to-last equality uses the fact that $h$ is finitely bounded and $|\mu|$ is an internal finitely bounded positive measure, whence Fact \ref{finiteintegral} applies.
\end{proof}

\subsection{Nonstandard hulls}

Fix an internal Hilbert space $H$.  We set $$\fin(H):=\{x\in H \ : \ \|x\|\in \fin(\starR)\}.$$  For $x\in \fin(H)$, we note that $x\mapsto \st(\|x\|)$ is a semi-norm on $\fin(H)$.  We also set $\operatorname{mon}(H):=\{x\in \fin(H)\ : \ \|x\|\approx 0\}$ and $\cirH:=\fin(H)/\operatorname{mon}(H)$, which is then a normed space under the norm induced by the above seminorm on $\fin(H)$.  If $x\in \fin(H)$, we denote its class in $\cirH$ by $\cirx$.  For $x,y\in \fin(H)$, note that $\langle x,y\rangle\in \fin(\starR)$ and thus the internal inner product on $H$ descends to an inner product on $\cirH$ defined by $\langle \cirx,\ciry\rangle:=\st(\langle x,y\rangle)$.  It follows from saturation that $\cirH$ is actually a Hilbert space, called the \textbf{nonstandard hull} of $H$.

If $B:H\to H$ is a finitely-bounded internal operator on $H$, meaning that the internal operator norm $\|B\|$ belongs to $\fin(\starR)$, then we note that $B(\fin (H))\subseteq \fin(H)$ and that $B$ descends to a well-defined bounded operator $\cirB:\cirH\to \cirH$ given by $\cirB(\cirx):={}^{\circ}(Bx)$, called the nonstandard hull of $B$.  Note that $\|\cirB\|=\st(\|B\|)$.

\subsection{Projection-valued Loeb measures}\label{projLoeb}

Suppose that $H$ is an internal Hilbert space and $P:{}^*\bor(\mathbb R)\to \cal B(H)$ is an internal projection-valued measure on ${}^*\bor(\mathbb R)$, meaning that for each $\Omega\in{}^*\bor(\mathbb R)$, there is an internal projection $P_\Omega$ on $H$ and this measure is merely finitely additive in the sense that if $\Omega_1,\ldots,\Omega_n\in {}^*\bor(\mathbb R)$ are pairwise disjoint and $\Omega=\bigcup_{i=1}^n \Omega_i$, then $P_\Omega x=\sum_{i=1}^n P_{\Omega_i}x$ for all $x\in H$.  Since internal projections are finitely bounded, we can consider their nonstandard hulls $\cirP\in \cal B(\cirH)$.  It is readily verified that each $\cirP$ is a projection on $\cirH$.  Moreover, for each $x,y\in H$, there is an associated internal complex measure $\mu^{x,y}$ on ${}^*\bor(\mathbb R)$ given by $\mu^{x,y}(\Omega)=\langle y,P_\Omega x\rangle$.  We know that $\mu^{x,y}$ is finitely bounded since $\|x\|,\|y\|\in \fin(\starR)$ and $|\mu^{x,y}|\leq \|x\|\ \|y\|$.  

By Lemma \ref{complexmeasure}, we have that $$\mu_L^{x,y}(\Omega)=\st(\mu^{x,y}(\Omega))=\st(\langle y,P_\Omega x\rangle)=\langle \ciry,\cirP_\Omega \cirx\rangle$$ for all $\Omega\in {}^*\bor(\mathbb R)$.  However, $\mu_L^{x,y}$ is defined on all of $\loeb(\starR)$ and this suggests defining $\cirP_{\Omega}$ for all $\Omega\in \loeb(\starR)$.  Note that, by familiar properties of Loeb measure, for $\Omega\in \loeb(\mathbb R)^*$, one has $$\mu_L^x(\Omega)=\sup_{\Omega'}\mu_L^x(\Omega')=\sup_{\Omega'}\langle \cirx, \cirP_{\Omega'}\cirx\rangle=\langle \cirx,\sup_{\Omega'}\cirP_{\Omega'}\cirx\rangle,$$ where all of the suprema are over the internally Borel subsets of $\Omega$.  In the last expression, the supremum is calculated in the lattice of projections on $\cal B(\cirH)$.  Motivated by this, for $\Omega\in \loeb(\starR)$, we thus define $\cirP_{\Omega}=\sup_{\Omega'}\cirP_{\Omega'}$, where $\Omega'$ ranges over the internally Borel subsets of $\Omega$.  One then has that $\mu^x_L(\Omega)=\langle \cirx,\cirP_{\Omega}\cirx\rangle$ for all $\cirx\in \cirH$.  Moreover, by the Polarization identity, it follows that $\mu^{x,y}_L(\Omega)=\langle \ciry,\cirP_\Omega\cirx\rangle$ for all $\cirx,\ciry\in \cirH$ and $\Omega\in \loeb(\starR)$.

With this set up, the following fact is fairly routine to prove; see \cite[Theorem 2]{raab}:

\begin{fact}
The map $\cirP:\loeb(\starR)\to \cal B(\cirH)$ is a projection-valued measure on $\loeb(\starR)$.
\end{fact}

\section{Revisiting the bounded case}\label{bounded}

In this section, we will give the proof of the Spectral theorem for \emph{bounded} self-adjoint operators, essentially following Moore's proof from \cite{moore}, but using the language established in the previous section.  Towards that end, we fix a bounded self-adjoint operator $A$ on a Hilbert space $\cal H$.  It is routine to find an internal, hyperfinite-dimensional Hilbert space $H$ such that $\cal H\subseteq H \subseteq {}^*\cal H$.  Let $P_{H}\in \cal B({}^*\cal H)$ denote the internal orthogonal projection map onto $H$ and set $$A^H:=P_{H}\circ (A\upharpoonright H):H\to H,$$ which is an internal linear map extending $A$.  Note that, for $x\in H$, one has $\|A^Hx\|=\|P_{H}(Ax)\|\leq \|P_{H}\|\|A\|\|x\|\leq \|A\|\|x\|$ and thus $A^H$ is finitely bounded with $\|A^H\|=\|A\|$.  Consequently, one may consider the nonstandard hull $\cirA^H$ of $A^H$.  It is clear that $\cal H$ embeds in $\cirH$ as a closed subspace via $x\mapsto \cirx$ and, under this identification, $\cirA^H$ is an extension of $A$ with $\|\cirA^H\|=\|A\|$.  

Moreover, note that $A^H$ is a self-adjoint operator on $H$.  Indeed, given $x,y\in H$, by transferring the fact that $A$ is self-adjoint, we have that
$$\langle A^Hx,y\rangle=\langle P_{H}Ax,y\rangle=\langle Ax,P_{H}y\rangle=\langle Ax,y\rangle=\langle x,Ay\rangle=\langle x,P_{H}(Ay)\rangle=\langle x,A^Hy\rangle.$$ It follows immediately that $\cirA^H$ is a self-adjoint operator on $\cirH$.  By Proposition \ref{inducedpvm}, in order to construct a spectral resolution of $A$, it suffices to construct a spectral resolution of $\cirA^H$.  (It is worth noting that Proposition \ref{inducedpvm} is much simpler to prove in the bounded case and Moore makes no fuss about this step in \cite{moore}.)

Since $A^H$ is a self-adjoint operator on a hyperfinite-dimensional space, by transferring the finite-dimensional case of Theorem \ref{maintheorem}, there is a hyperfinite set $\lambda_1,\ldots,\lambda_n$ of eigenvalues of $A^H$ such that, setting $P_i$ to be the internal projection onto the eigenspace of $A^H$ corresponding to $\lambda_i$, the internal spectral resolution of $A^H$ is given by $A^H=\sum_{i=1}^n \lambda_i P_{i}$.  Using this, we can define an internal projection-valued measure $P:{}^*\bor(\mathbb R)\to \cal B(H)$ by defining $P(\Omega)=\sum_{\lambda_i\in \Omega}P_{i}$.  We note that $P$ is supported on $\sigma(A^H)\subseteq {}^*[-\|A\|,\|A\|]$, that is, $P(\sigma(A^H))=I$.  (Here, $\sigma(A^H)$ is the internal spectrum of $A^H$, namely the set of eigenvalues of $A^H$.)  As discussed in Subsection \ref{projLoeb}, we obtain a genuine projection-valued measure $\cirP:\loeb(\starR)\to \cal B(\cirH)$.

In order to obtain a projection-valued measure on $\bor(\mathbb R)$, we pushforward $\cirP$ using the standard part map.  More precisely, we consider the projection-valued measure $Q:=\bst_*\cirP:\bor(\mathbb R)\to \cal B(\cirH)$.  For $\cirx,\ciry\in \cirH$, we let $\nu^{x,y}$ denote the complex measure associated to $Q$, $\cirx$, and $\ciry$.  (Note that this abuse of notation is acceptable since the complex measures corresponding to different representatives of $\cirx$ and $\ciry$ coincide.)  Moreover, each $\nu^{x,y}$ is supported on $[-\|A\|,\|A\|]$.

The Spectral theorem for bounded self-adjoint operators follows from the following:

\begin{thm}\label{boundedres}
$Q$ is a spectral measure for $\cirA^H$.
\end{thm}

\begin{proof}
Since each $\nu^{x,y}$ has compact support, we have that $\int_{\mathbb R}\id dQ$ is a bounded operator on $\cirH$.  By the Polarization identity, it suffices to show that $\langle \cirx,\cirA^H (\cirx)\rangle=\langle \cirx,(\int_{\mathbb R}\id dQ)\cirx \rangle$ for all $\cirx\in \cirH$ of unit norm.  For convenience, set $K:=\|A\|$.  We calculate
$$\langle \cirx, \cirA^H(\cirx)\rangle=\st\left(\int_{\starR}\id d\mu^x\right)=\st\left(\int_{\starR}\chi_{{}^*[-K,K]}\id d\mu^x\right)=\int_{\starR}\chi_{{}^*[-K,K]}\st d\mu^x_L.$$  Note that the last equality follows from Fact \ref{finiteintegral}.  Since ${}^*[-K,K]$ has full $\mu_L$-measure, this latter integral equals $$\int_{\starR}\bst d\mu^x_L=\int_{\mathbb R}\id d\nu^x=\left\langle \cirx,\left(\int_{\mathbb R}\id dQ\right)\cirx\right\rangle,$$ as desired.
\end{proof}

% Add remark that motivation for pushing forward using standard part is that every element of the spectrum of $A$ is infinitely close to an eigenvalue of $B$.

\section{Nonstandard hulls of internally bounded self-adjoint operators}

In this section, we fix an internally bounded (but not necessarily finitely bounded) self-adjoint operator $B:H\to H$ on some internal Hilbert space $H$ with internal projection-valued measure $P:{}^*\bor(\mathbb R)\to \cal B(H)$.  As in the preceding section, we consider the genuine projection-valued measure $\cirP:\loeb(\starR)\to \cal B(\cirH)$ with pushforward $Q:=\bst_*P:\bor(\mathbb R)\to \cal B(\cirH)$.  Motivated by Theorem \ref{boundedres}, we make the following definition:

\begin{defn}
The \textbf{nonstandard hull} of $B$ is the unbounded operator $\cirB$ on $\cirH$ given by $\cirB:=\int_{\mathbb R}\id dQ$.
\end{defn}

In other words, we are \emph{defining} the nonstandard hull $\cirB$ of $B$ by specifying its spectral resolution.  Theorem \ref{boundedres} shows that the two definitions of $\cirB$ agree in the case that $B$ is finitely bounded.  (Technically speaking, Theorem \ref{boundedres} dealt with the case that $B=\cirA^H$ for some bounded self-adjoint operator $A$ and some hyperfinite-dimensional space $H$ with $\cal H\subseteq H\subseteq {}^*\cal H$, but the proof works equally well for any internal finitely bounded self-adjoint operator equipped with its internal spectral resolution.)

By Lemma \ref{pushforward}, we may alternatively write $\cirB=\int_{\starR}\bst d\cirP$.  (This latter expression is how Raab defines $\cirB$ in \cite{raab}.)  Note also that $$D(\cirB)=\left\{\cirx\in \cirH \ : \ \int_{\mathbb R}t^2d\nu^x(t)<\infty\right\}=\left\{\cirx\in \cirH \ : \ \int_{\starR}\bst(t)^2d\mu^x_L(t)<\infty\right\}.$$  

The main goal in this section is Lemma \ref{nshullformula} below.  The content of that lemma is stated without proof in \cite{raab}.  We believe it is best to give a detailed proof of this important result.  Towards that end, we prove the following preliminary lemma.

\begin{lem}\label{proj}
For all $\cirx\in D(\cirB)$, we have $\cirB\cirx\in \cirP_{\fin(\starR)}(\cirH)$.
\end{lem}

\begin{proof}
Fix $\ciry\in \cirH$.  Then 
$$\langle \ciry,\cirP_{\fin(\starR)}\cirB\cirx\rangle=\int_{\starR}\chi_{\fin(\starR)}\bst d\mu^{x,y}_L=\int_{\starR}\bst d\mu^{x,y}_L=\langle \ciry,\cirB\cirx\rangle.$$
The desired result now follows.
% $$\langle \ciry,\cirP_{\fin(\starR)}\cirB\cirx\rangle=\sup_{n\in \mathbb N}\langle \ciry,\cirP_{[-n,n]^*}\cirB\cirx\rangle=\sup_{n\in \mathbb N}\int_{\mathbb R}\chi_{[-n,n]}zd\nu^{x,y}=\langle \ciry,\cirB\cirx\rangle.$$  The desired result now follows.  CANT APPLY SUP TRICK IN THIS RESULT AND NEXT BECAUSE THEY ARE COMPLEX MEASURES.
\end{proof}

As Raab points out, the following formula relates $\cirB\cirx$ and ${}^{\circ}(Bx)$ in the case that $x$ and $Bx$ both belong to $\fin(H)$.

\begin{lem}\label{nshullformula}
Suppose that $x$ and $Bx$ both belong to $\fin(H)$.  Then $\cirx\in D(\cirB)$ and $\cirB\cirx=\cirP_{\fin(\starR)}{}^{\circ}(Bx)$.
\end{lem}

\begin{proof}
Since $\int_{\starR}t^2d\mu^x(t)=\|Bx\|^2\in \fin(\starR)$, Fact \ref{Lindstrom} implies that $\bst^2$ is $\mu^x_L$-integrable, whence $\cirx\in D(\cirB)$.
% Note that
% $$\int_{\mathbb R}z^2d\nu^x=\int_{\mathbb R}z^2d(\st_*\mu^x_L)=\int_{\starR}\st^2d\mu_L^x$$
% and $\|Bx\|^2=\int_{\starR}z^2d\mu^x$.  By Lindstrom's theorem, we have that $\int_{\starR}\st^2d\mu^x_L$ is finite, whence $\cirx\in D(\cirB)$.  

Now fix $\ciry\in\cirH$.  We then have
\begin{align}
    \langle\cirP_{\fin(\starR)}\ciry,{}^{\circ}(Bx)\rangle&=\langle \ciry,\cirP_{\fin(\starR)}{}^{\circ}(Bx)\rangle \notag \\
    &=\lim_{n\to \infty}\langle \ciry,\cirP_{{}^*[-n,n]}{}^{\circ}(Bx)\rangle \notag \\ 
    &=\lim_{n\to \infty}\st\langle y,P_{{}^*[-n,n]}(Bx)\rangle \notag \\
    &=\lim_{n\to \infty}\st\left(\int_{{}^*[-n,n]}\id d\mu^{x,y}\right) \notag \\
    &=\lim_{n\to \infty}\st\left(\int_{{}^*[-n,n]}th(t)d|\mu^{x,y}|(t)\right) \notag \\
    &=\lim_{n\to \infty}\int_{{}^*[-n,n]}\st(t)\st(h(t))d(|\mu^{x,y}|_L)(t) \notag \\
    &=\int_{\starR}\chi_{\fin(\starR)}\cdot \bst\cdot \st(h)d(|\mu^{x,y}|_L) \notag \\
    &=\int_{\starR}\bst d\mu^{x,y}_L\notag \\
    &=\int_{\mathbb R}\id d\nu^{x,y}\notag \\
    &=\langle \ciry,\cirB\cirx\rangle\notag \\
    &=\langle \cirP_{\fin(\starR)}\ciry,\cirB\cirx\rangle.\notag \\ \notag
\end{align}
% $$\langle\cirP_{\fin(\starR)}\ciry,{}^{\circ}(Bx)\rangle=\langle \ciry,\cirP_{\fin(\starR)}{}^{\circ}(Bx)\rangle=\lim_{n\to \infty}\langle \ciry,\cirP_{[-n,n]^*}{}^{\circ}(Bx)\rangle.$$
% Continuing, this equals
% $$\lim_{n\to \infty}\st\langle y,P_{[-n,n]^*}(Bx)\rangle=\lim_{n\to \infty}\st\int_{[-n,n]^*}zd\mu^{x,y}=\lim_{n\to \infty}\st\int_{[-n,n]^*}zhd|\mu^{x,y}|.$$  
In the above, $h=\frac{d|\mu^{x,y}|}{d\mu^{x,y}}$.  We note that the second equality follows from the fact that the sequence of projections $\cirP_{{}^*[-n,n]}$ converges in the strong operator topology to $\cirP_{\fin(\starR)}$, the sixth equality follows the fact that $h$ is internally bounded and Fact \ref{finiteintegral}, the seventh equality follows from the Dominated Convergence theorem together with the fact that $\bst\in L^1(|\mu_L^{x,y}|)$  as $\cirx\in D(\cirB)$, while the final equality follows from Lemma \ref{proj}.

% Since $h$ is internally bounded, by applying Fact \ref{finiteintegral} above to the real and imaginary parts of $zh$, we get that this equals
% $$\lim_{n\to \infty}\int_{[-n,n]^*}\st(z)\st(h)d(|\mu^{x,y}|_L).$$
% Since $\cirx\in D(\cirB)$ and $\cirB=\int_{\mathbb{R}^*}\bst d\cirP$, we have that $\bst\in L^1(|\mu^{x,y}_L|)$.  Thus, by the Dominated Convergence Theorem, the last display equals
% $$\int_{\starR}\chi_{\fin(\starR)}\bst(z)\st(h)d(|\mu^{x,y}|_L)=\int_{\starR}\bst d\mu^{x,y}_L=\int_{\mathbb R}zd\nu^{x,y}=\langle \ciry,\cirB\cirx\rangle.$$

% =\lim_{n\to \infty}\int_{[-n,n]^*}\st(z)d\mu^{x,y}_L=\lim_{n\to \infty}\int_{[-n,n]}zd\nu^{x,y}.$$  Since $\cirx\in D(\cirB)$, we have that $z\in L^1(\nu^{x,y})$ and thus by Dominated Convergence the latter limit equals $\int_{\mathbb R}zd\nu^{x,y}=\langle \ciry,\cirB\cirx\rangle$.  
%This latter quantity equals $\langle \cirP_{\fin(\starR)}\ciry,\cirB\cirx\rangle$ by the previous lemma.

% \sup_{n\in \mathbb N}\langle \ciry,\cirP_{[-n,n]^*}{}^{\circ}(Bx)\rangle.$$. Continuing, this equals
% $$\sup_{n\in \mathbb N}\st\langle y,P_{[-n,n]^*}(Bx)\rangle=\sup_{n\in \mathbb N}\st\int_{[-n,n]^*}zd\mu^{x,y}=\sup_{n\in \mathbb N}\int_{[-n,n]^*}\st(z)d\mu^{x,y}_L.$$. Continuing, this equals
% $$\sup_{n\in \mathbb N}\int_{[-n,n]}zd\nu^{x,y}=\int_{\mathbb R}zd\nu^{x,y}=\langle \ciry,\cirB\cirx\rangle=\langle \cirP_{\fin(\starR)}\ciry,\cirB\cirx\rangle.$$
% The last equality follows from the previous lemma.  
We conclude that ${}^{\circ}(Bx)-\cirB\cirx\perp \cirP_{\fin(\starR)}(\cirH)$.  By Lemma \ref{proj}, the statement of the current lemma follows.
\end{proof}

\section{The proof of the unbounded case}

In this section, we complete the proof of Theorem \ref{maintheorem}.  Towards that end, we fix a symmetric operator $A$ on $\cal H$ and let $H$ be a hyperfinite-dimensional subspace of ${}^*\cal H$ such that $D(A)\subseteq H\subseteq {}^*D(A)$.  As in Section \ref{bounded}, we set $A^H:=P_{H}\circ(A\upharpoonright H)$, where $P_{H}\in \cal B({}^*\cal H)$ is the internal orthogonal projection.  Arguing as in Section \ref{bounded}, one sees that $A^H$ is symmetric.  It follows that $A^H$ is internally bounded and self-adjoint.  (Here we are transferring the fact that a symmetric operator whose domain is the entire Hilbert space is necessarily bounded and self-adjoint; see \cite[X.2.4(c)]{conway}.)

In the proof of the next result, we need to remark that if $x\in \cal H$, then $P_{H}x\approx x$.  To see this, simply note that, since $D(A)$ is dense in $\cal H$, for each (standard) $n\in \mathbb N$, there is $y\in D(A)\subseteq H$ such that $\|x-y\|<\frac{1}{n}$; thus, by overflow, there is $N\in \mathbb N^*\setminus \mathbb N$ for which there is $z\in H$ satisfying $\|x-z\|<\frac{1}{N}$.  Consequently, $\|x-P_{H}x\|\leq \|x-z\|<\frac{1}{N}$, proving the claim. 

% Now
% $$\|\cirP_{\fin(\starR)}{}^{\circ}(Bx)\|^2=\sup_{n\in \mathbb N}\|\cirP_{[-n,n]^*}\ ^{\circ}(Bx)\|^2=\sup_{n\in \mathbb N}\st\langle x,BP_{[-n,n]^*}(Bx)\rangle.$$
% But $$\st(\langle x,BP_{[-n,n]^*} Bx\rangle)=\st(\int_{\starR} z^2\chi_{[-n,n]^*} d\mu^x)=\int_{\starR}\st^2\chi_{ [-n,n]^*}d\mu^x_L=\int_{\mathbb R}z^2\chi_{[-n,n]}d\nu^x.$$ Taking supremums over $n$ we get $\|\cirP_{\fin(\starR)}{}^{\circ}(Bx)\|^2=\|\cirB\cirx\|^2$.  Since $\cirB\cirx\in \cirP_{\fin(\starR)}(\cirH)$, we get the desired result.

% Now fix bounded Borel $\Omega\subseteq \mathbb R$ and note that $\langle \cirx,\cirP_{\st^{-1}(\Omega)}\cirx\rangle=\st(\langle x,P_{\st^{-1}(\Omega)}x\rangle)=\st(\int_{\st^{-1}(\Omega)}zd\mu^x)=\int_{\st^{-1}(\Omega)}\st d\mu^x_L$.  (Check this last step.  Maybe just use $\Omega=[-n,n]$.)  Consequently
% $$\langle \cirx,\cirP_{\fin(\starR)}\cirx\rangle=\sup_{\Omega}\langle \cirx,\cirP_{\st^{-1}(\Omega)}\cirx\rangle=\sup_\Omega \int_{\st^{-1}(\Omega)}\st d\mu^x_L=\sup_\Omega \int_{\bar \Omega}zd\nu^x.$$. This equals
% $$\int_{\mathbb R}zd\nu^x=\langle \cirx,\cirB
%\end{proof}

As we saw in Section \ref{bounded}, the statement of the next result is clear in the case that $A$ is bounded.  However, since $\cirA^H$ was defined in terms of its spectral resolution, the statement is far from obvious.

\begin{thm}
In the set-up above, we have $A\subseteq \cirA^H$.
%$D(A)\subseteq D(\cirB)$ and $\cirB x=Ax$ for all $x\in D(A)$.
\end{thm}

\begin{proof}
Fix $x\in D(A)$.  By the remark preceding this proof, we have $A^Hx\approx Ax$, whence Lemma \ref{nshullformula} implies that $x\in D(\cirA^H)$.  Now
$$\langle x,\cirA^H(x)\rangle=\int_{\mathbb R}\id d\nu^x=\int_{\starR}\bst d\mu^x_L=\st\left(\int_{\starR} \id d\mu^x\right)=\st\left(\langle x,A^H x\rangle\right)=\langle x,{}^{\circ}(A^Hx)\rangle.$$ Note that the third equality follows from Fact \ref{Lindstrom}.  Using the Polarization identity, we get that $\langle y,\cirA^H(x)\rangle=\langle y,{}^{\circ}(A^Hx)\rangle$ for all $x,y\in D(A)$.  Using that $D(A)$ is dense in $\cal H$, we see that
\begin{align}
    \|Ax\|&=\|{}^{\circ}(A^Hx)\| \notag \\
    &=\sup_{y\in D(A)_1}|\langle y,{}^{\circ}(A^Hx)\rangle| \notag \\
    &=\sup_{y\in D(A)_1}|\langle y,\cirA^H(x))\rangle| \notag \\
    &\leq \sup_{y\in \cirH_1}|\langle y,\cirA^H(x)\rangle| \notag \\
    &=\|\cirA^H(x)\|.\notag 
\end{align}

Here, the subscript $1$ indicates that we are only suping over the unit balls of the corresponding spaces.  The displayed inequality together with Lemma \ref{nshullformula} implies the desired result. 
\end{proof}

In the context above, we refer to $\cirA^H$ as a \textbf{hull extension} of $A$.  (Raab calls this a nonstandard extension of $A$ in \cite{raab}, but we find this confusing as it is not literally the same object as the actual nonstandard extension of $A$.)  The previous theorem justifies the word ``extension'' in the term hull extension.

The proof of Theorem \ref{maintheorem} now follows.  Indeed, we can choose $H$ above to be hyperfinite-dimensional, whence we know that the spectral resolution for $A^H$ already exists by transfer as in our proof of the bounded case.  If $A$ is assumed to be self-adjoint, then $A=\cirA^H|\cal H$ by Proposition \ref{selfadjointext} and thus the existence of the spectral resolution of $A$ follows from Proposition \ref{inducedpvm}.

\begin{remark}
The above discussion yields a nonstandard proof of the well-known fact that every symmetric operator has a self-adjoint extension on a larger Hilbert space.  
\end{remark}

Using the above analysis, we can also obtain a nonstandard criterion for when a symmetric operator $A$ on $\cal H$ has a self-adjoint extension on $\cal H$ itself:  

\begin{cor}
If $A$ is a symmetric operator on $\cal H$, then there is a self-adjoint operator on $\cal H$ extending $A$ if and only if there is a hull extension $\cirA^H$ of $A$ such that $\cal H$ is reducing for $\cirA^H$.
\end{cor}

%By Lemma \ref{suffices}, we see that $O\circ P$ is the spectral resolution of $A$ in case $A$ is self-adjoint, finishing the proof of the Spectral theorem for unbounded self-adjoint operators..


\begin{thebibliography}{}
\bibitem{conway} J. Conway, \textit{A course in functional analysis}, Graudate Texts in Mathematics \textbf{96} (2007), Springer-Verlag.
\bibitem{gold} I. Goldbring and S. Walsh, \emph{An invitation to nonstandard analysis and its recent applications}, Notices of the AMS, \textbf{66} (2019), 842-851.
\bibitem{moore} L. C. Moore, Jr., \textit{Hyperfinite extensions of bounded operators on a separable Hilbert space}, Transactions of the American Mathematical Society \textbf{218} (1976), 285-295.
\bibitem{raab} A. Raab, \textit{An approach to nonstandard quantum mechanics}, Journal of Mathematical Physics \textbf{45} (2004), 4791-4809.
\bibitem{ross} D. Ross, \textit{Loeb measure and probability}, Nonstandard Analysis:  Theory and Applications, NATO ASI Series \textbf{493} (1997), 91-120.
\bibitem{rudin} W. Rudin, \textit{Real and complex analysis}, McGraw-Hill Series in Higher Mathematics (1986).
\end{thebibliography}
\end{document}